\documentclass[14pt]{amsart}
\usepackage[T1]{fontenc}
\usepackage[utf8]{inputenc}
\usepackage{eucal}
\address{\newline{\normalsize Laboratory of AGHA, Moscow Institute of Physics and Technology, 9 Institutskiy per., Dolgoprudny,
Moscow Region, 141701, Russia}
\newline{\it E-mail address}: karzhemanov.iv@mipt.ru}

\usepackage{listings}

\usepackage{amscd,amsthm,amsmath,amssymb}
\usepackage{textcomp}
\usepackage{graphicx}
\usepackage[matrix,arrow]{xy}

\usepackage[unicode]{hyperref}

\makeatletter\@addtoreset{equation}{section}\makeatother

\makeatletter\@addtoreset{subsection}{equation}\makeatother

\newtheorem{theorem}[equation]{Theorem}
\newtheorem{prop}[equation]{Proposition}
\newtheorem{lemma}[equation]{Lemma}
\newtheorem{cor}[equation]{Corollary}

\newtheorem{theorem-definition}[equation]{Theorem-definition}

\theoremstyle{definition}

\theoremstyle{remark}
\newtheorem{remark}[equation]{Remark}

\providecommand{\doi}[1]{\href{https://#1}{\tt https:#1}}

\newcommand{\com}{\mathbb{C}}

\newcommand{\p}{\mathbb{P}}

\newcommand{\cel}{\mathbb{Z}}

\newcommand{\map}{\longrightarrow}

\thanks{{\it MS 2010 classification}: 14E05, 14D05, 68-04}


\begin{document}

\title{Computations and ML for surjective rational maps}

\author{Ilya Karzhemanov}

\begin{abstract}
The present note studies \emph{surjective rational endomorphisms}
$f: \p^2 \dashrightarrow \p^2$ with \emph{cubic} terms and the
indeterminacy locus $I_f \ne \emptyset$. We develop an
experimental approach, based on some Python programming and
Machine Learning, towards the classification of such maps; a
couple of new explicit $f$ is constructed in this way. We also
prove (via pure projective geometry) that a general non-regular
cubic endomorphism $f$ of $\p^2$ is surjective if and only if the
set $I_f$ has cardinality at most $6$.
\end{abstract}

\sloppy

\maketitle

\bigskip

\section{Introduction}
\label{section:introd}

\refstepcounter{equation}
\subsection{The set-up}
\label{subsection:introd-1}

Let $X$ be a complex projective variety and $f: X \dashrightarrow
X$ its rational endomorphism. Denote by $I_f \subset X$ the
(closed) indeterminacy locus of $f$. Then $f$ is called
\emph{surjective} if the induced morphism $X \setminus I_f \map X$
is onto. Such maps were introduced and studied --- from the
algebro-geometric point of view --- in the paper \cite{Ilya-Ilya}
(with $X$ being a projective space $\p^n$). The initial motivation
for this study came from the optimal control theory and
mathematical genetics. Later a relation of surjective rational
maps with algebraic statistics and quantization has also been
realized (see \cite{Ilya-ML} and \cite{Ilya-Quant}). Another
classical subject here is holomorphic dynamics: the given $f$ is
very much akin to a \emph{polynomial-like map} (see
\cite[Ch.\,2]{Dinh-Syb}); hence one could try to define dynamical
degrees of $f$ (with respect to some K\"ahler form on $X$) and
(re)prove standard relations between them and the entropy (cf.
\cite[Remark 1.2]{Ilya-Ilya}).

In the present note, we concentrate on the behavior of surjective
maps in families, aiming to show that for certain rational maps
the surjectivity is a \emph{generic property}. Namely, let us take
$X = \p^2$, so that in projective coordinates any of its rational
endomorphisms looks like this:
$$
f = [f_0: f_1: f_2]: \p^2 \dashrightarrow \p^2,
$$
where $f_i$ are mutually coprime homogeneous polynomials of some
degree $d$. These polynomials span a plane $\Pi$ in the linear
system $|\mathcal{O}_{\p^2}(d)|$ of all degree $d$ curves on
$\p^2$.\footnote{All notions and facts from algebraic geometry
used in this paper can be found in e.g. \cite{Dol-CAG}.} Note that
$I_f \subset \p^2$ is given by equations $f_0 = f_1 = f_2 = 0$.
Then the morphism $f: \p^2 \setminus I_f \map \p^2$ can be defined
in the following coordinate-free way (\emph{Kodaira's
construction}): for any point ${\bf x} \in \p^2 \setminus I_f$,
consider the pencil $\ell \subset \Pi$ of all linear combinations
of $f_i$, which vanish at ${\bf x}$; now identify $\p^2$ with the
dual plane $\Pi^*$ of lines and put
$$
f: {\bf x} \mapsto \ell \in \p^2 = \Pi^*.
$$
So the term \emph{generic}, when applied to the map $f$, simply
means that the corresponding plane $\Pi \subset
|\mathcal{O}_{\p^2}(d)|$ is such (or, if being more formal, one
may state genericity in terms of points on the relevant
Grassmannian).

Further, the $0$-scheme $I_f$ varies together with $f$, of course,
and one can always achieve $I_f = \emptyset$ for appropriate
$f_i$. Hence the map $f$ becomes \emph{regular} and its
surjectivity is straightforward. Thus, in order to have a fruitful
problem here, let us fix a (generic) set $I_f$, consider the
maximal linear system $\Lambda \subset  |\mathcal{O}_{\p^2}(d)|$
with the base locus $\text{Bs}(\Lambda) = I_f$, and treat only
those planes $\Pi$, which belong to $\Lambda$. In this setting,
the generality problem has been solved completely in
\cite{Kul-Zhd} for the \emph{quadratic} case, i.e. when $d = 2$.
More precisely, \cite[Theorem 2]{Kul-Zhd} asserts that for every
given cardinality $\#I_f$ the set of all $\Pi \subset \Lambda$,
defining surjective maps, is \emph{Zariski open} (including the
set $\emptyset$).

Presently, we consider the next, \emph{cubic}, case of $d = 3$.
Our aim is two-fold. Firstly, we extend the main result of
\cite{Kul-Zhd} by showing that all such surjective $f$, with fixed
generic $I_f$, form an open set (see Section~\ref{section:prof}
below). Secondly, we develop an experimental approach towards a
complete description of (cubic) surjective rational maps, which
also yields new \emph{explicit} examples of them (see
Section~\ref{section:ex-calc}).

\begin{remark}
\label{remark:surj-robust} Surjectivity property seems like
another \emph{robust} substitute for the term ``generic'' among
many others that one comes across in classical algebraic geometry.
For instance, given a smooth conic $(q = 0) \subset \p^2$ the set
of triples of linear forms $l_i$, satisfying $q =
\displaystyle\sum_i l_i^2$, can be compactified into the
\emph{variety of self-polar triangles}
--- this is a codimension $3$ linear section of the Grassmannian
$G(3,5)$, --- and the self-polarity condition can be effectively
verified (see \cite[{\bf 2.1.3}]{Dol-CAG}). Another example is the
non-gradient property for a given quadratic endomorphism of $\p^3$
in terms of smoothness of the Jacobi surface (see
\cite{Ilya-non-grad}). Coming back to our setting, a rational map
$f: \p^2 \dashrightarrow \p^2$ is surjective iff there is no
pencil $\ell \in \Pi^*$ as above, with the set-theoretic equality
$\text{Bs}(\ell) = \text{Bs}(\Pi)$ (we will call such $\ell$ an
\emph{unruly pencil}). The latter property can be also checked
explicitly via the direct calculation (cf. the code in
{\ref{subsection:ex-calc-1}} below).
\end{remark}

\refstepcounter{equation}
\subsection{del Pezzo surfaces}
\label{subsection:introd-2}

For what follows, let us fix $I_f$, $\Lambda \subset
|\mathcal{O}_{\p^2}(3)|$ and $\Pi \subset \Lambda$ as in
{\ref{subsection:introd-1}}. We also assume $I_f$ to be
\emph{generic}. In particular, we have $I_f = \{P_1,\ldots,P_{9 -
\delta}\}$ for some $\delta \in \cel$, where the points $P_i$ are
in general position (no three on a line and no six on a conic); it
is easy to see via the parameter count that $2 \le \delta \le 8$
(see {\ref{subsection:prof-1}} below).

Now consider the blow-up $\varphi: X \map \p^2$ of all the $P_i$.
Then $X$ is the celebrated del Pezzo surface of degree $\delta$
(see \cite[Ch.\,8]{Dol-CAG}). It is well-known that $\delta \ge 3$
iff the linear system $|-K_X|$ defines the \emph{anticanonical
embedding} $X = X_{\delta} \subset \p^{\delta} = \Lambda^*$ of $X$
as a degree $\delta$ surface. This allows one to treat the map
$f$, for $\delta \ge 3$, as a linear projection of $X$ onto the
plane $\p^2 = \Pi^*$ (cf. Figure~\ref{fig-2} below). Then the
surjectivity of $f$ is established by a simple geometric argument
(see Theorem~\ref{theorem:main}). In the remaining $\delta = 2$
case, there is an explicitly constructed unruly pencil $\ell$, so
that $f$ is not surjective (see
Corollary~\ref{theorem:bad-pencil-pi-cor}).

\begin{remark}
\label{remark:torsor} One interesting fact about smooth del Pezzo
surface $X$ with $\delta \ge 3$ is that its \emph{universal
torsor} admits an embedding into a homogeneous space $G \slash P$.
Here $G$ a semisimple algebraic group with parabolic subgroup $P$
corresponding to the root system of rank $9 - \delta$ in the
Picard lattice $\text{Pic}(X)$ (see \cite{Serg-Skor}). It would be
interesting to contemplate this phenomenon in terms of the above
surjectivity result for $f$.
\end{remark}

Surjective rational maps and del Pezzo surfaces have already
interacted with each other (see \cite{Ilya-surj-dP}). Let us make
one step further and indicate a relation between (surjective)
cubic maps $f$ and \emph{rational elliptic surfaces}. First of
all, recall that to an arbitrary surjective endomorphism $f: \p^n
\dashrightarrow \p^n, n \ge 1$, one associates a vector bundle
$\mathcal{E} \simeq \Omega^1_{\p^n}$ together with an inclusion
$f^*(\mathcal{E}) \subseteq \mathcal{E} \otimes
\mathcal{O}_{\p^n}(-d+1)$, where $f^*$ is a properly defined
pull-back (see \cite{Ilya-ML}). Thus one may regard every such $f$
as a certain connection $\nabla$ on the (``Schwarzenberger-like'')
bundle $\mathcal{E}$ (cf. \cite[{\bf 2.2.2}]{Dol-CAG}). Secondly,
in our case of $n=2,d=3$ projectivization $\p(\mathcal{E})$ is
identified (via $f$) with a family of \emph{elliptic pencils}
$\ell$, which yields a \emph{flat family} of rational elliptic
surfaces $\frak{X} \map \p^2$. This $\frak{X} \slash \p^2$
resembles one of the \emph{Painlev\'e families} (see e.g.
\cite{Saito-Umemura}, \cite{Kajiwara-et-al}), where now
compatibility with $\nabla$ plays the role of the isomonodromy
condition.

In view of the present discussion, it is tempting to look for
(hidden) symmetries associated with $f$, which one may construct
in terms of sublattices in $\text{Pic}(X)$ for the fibers $X
\subset \frak{X}$ (cf. Remark~\ref{remark:torsor}). Such instances
as modular transformations, or q-mutations as their quantum analog
(see \cite{Usnich}), plus various types of dualities and
correspondences (like the S-duality in the del Pezzo case --- see
e.g. \cite{King}), are awaiting for their way towards surjective
rational maps.\footnote{Another example of how the geometry of del
Pezzo surfaces can be applied to (wild) automorphisms of
threefolds is contained in \cite{Ilya-dP-iso}.} Finally, we can
not help mentioning the \emph{spectral data} approach in this
context (cf. \cite{Hitchin}), which also has some interactions
with surjectivity and quantization (see \cite{Ilya-Quant}).

\refstepcounter{equation}
\subsection{Experiments}
\label{subsection:introd-3}

The use of various neural networks is so ubiquitous in today's
scientific research that it is difficult to bypass this subject.
Presently, we have developed an experimental approach, whose
ultimate goal is an effective description/classification of all
surjective rational endomorphisms of $\p^2$ (not necessarily
having $d = 3$). It goes as follows. First one inspects all
dominant maps $\p^2 \dashrightarrow \p^2$ over a small finite
field $\mathbb{F}_p$ and generates a table (aka data set) by the
rule (see {\ref{subsection:ex-calc-1}} below): $f \mapsto {\bf 1}$
(surjective) or ${\bf 0}$ (not surjective). After this one learns
the table via NN and generates a function $f \mapsto \Psi(f)$ (a
``measure of surjectivity''), now defined for \emph{all $f$ over
$\com$}, aiming to predict whether a given map is surjective or
not (see {\ref{subsection:ex-calc-2}}).\footnote{One may regard
this reasoning as a kind of ``numerical Hasse principle''. It is
worth to compare this principle with various \emph{(topological)
persistency} phenomena that one encounters in applied sciences
(see e.g. the survey \cite{Carlsson}).} Note that this line of
thought is not new and resembles the one of \emph{Natural
Intelligence} as demonstrated in the wonderful book
\cite{Arnold-finite-dyn}. There, a lot of numerical experiments
over small $\mathbb{F}_p$'s, concerning finite dynamics,
statistics of discrete functions and finite projective structures,
were carried out (by hand!), after what fruitful conjectures got
formulated and even proved in some cases.

Further, as is typical to the experiments, their output is often
confusing and should be taken with some proviso. Our small project
is not an exception. The major issue is that the function $\Psi$
does not take \emph{integer} values and hence it is not clear,
given the value $\Psi(f)$, how ``far'' the map $f$ is from being
surjective (this puzzle creates, obviously, a groundwork for
future research). Furthermore, the behavior of $\Psi$ shows that
the openness property of surjective cubic maps, discussed earlier
in {\ref{subsection:introd-2}}, \emph{does not} hold over
$\mathbb{F}_2$. Still, the outcome of our calculations was not
empty, and we managed to construct a couple of non-trivial
surjective planar maps with $d = 3$; here we could not completely
avoid the usage of NI (see Propositions~\ref{theorem:4-dp-surj}
and \ref{theorem:3-dp-surj}, as well as the whole
Section~\ref{section:ex-calc} for further discussion and details).

In conclusion, let us refer to the survey \cite{Gukov-et-al} for
an excellent overview of how the Machine Learning applies to
various questions of pure science, including extremal Calabi--Yau
metrics, knots and string theory. We hope that surjective rational
maps will make their path to these advanced subjects in future.

\bigskip

\section{Surjectivity results}
\label{section:prof}

\refstepcounter{equation}
\subsection{The case $\delta \le 2$}
\label{subsection:prof-1}

Note that passing through the (general) points $P_1,\ldots,P_{9 -
\delta}$ as above imposes $9 - \delta$ independent linear
conditions on the linear system $|\mathcal{O}_{\p^2}(3)| \simeq
\p^9$. Hence we get $0 \le \delta \le 9$. Furthermore, if $\delta
\le 1$, then $\dim \Lambda < 2$ and there is no plane $\Pi \subset
\Lambda$.

Now let us assume that $\delta = 2$. Then we obtain $\Pi =
\Lambda$ and the following holds:

\begin{prop}
\label{theorem:bad-pencil-pi} There exists an unruly pencil $\ell
\subset \Pi$.
\end{prop}

\begin{proof}
Note that by the parameter count as above there is a cubic $C \in
\Pi$, which has $P_1$ as its inflection point. In turn, this $C$
is contained in a pencil $\ell$ of cubics from $\Pi$, sharing a
common tangent with $C$ at $P_1$. Take some $Z \in
\ell\setminus\{C\}$. Then we have
$$
C \cap Z = \{P,P_1,P_2,\ldots,P_7\}
$$
for some point $P$. Now, if $P \ne P_1$, then again the parameter
count gives $\dim \ell = 9 - 6 - 2 - 1 = 0$ (``$1$'' is for
passing through extra $P$), a contradiction. Hence $P = P_1$ (see
Figure~\ref{fig-1}) and $\ell$ is unruly.

\bigskip

\begin{figure}[h]
\includegraphics[scale=1.1]{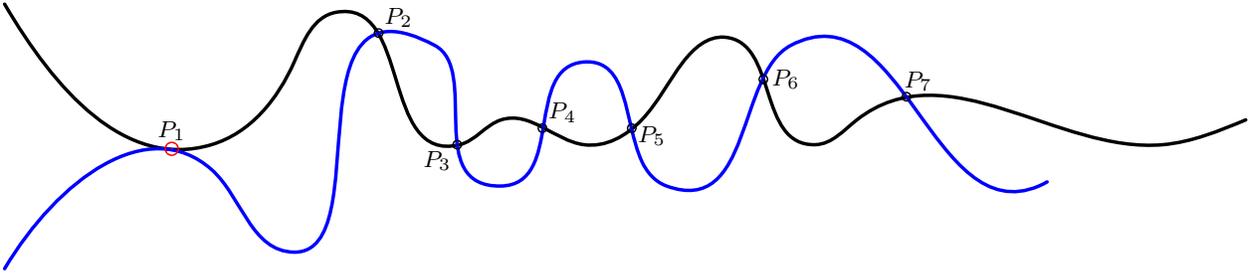}
\caption{Two curves $C$ and $Z$}\label{fig-1}
\end{figure}

\end{proof}

From Proposition~\ref{theorem:bad-pencil-pi} and
Remark~\ref{remark:surj-robust} we deduce the following:

\begin{cor}
\label{theorem:bad-pencil-pi-cor} Generic rational endomorphism
$f$ of the plane, having cubical components and $\#I_f = 9 -
\delta \ge 7$, is \emph{not surjective}.
\end{cor}

\refstepcounter{equation}
\subsection{The case $\delta \ge 3$}
\label{subsection:prof-2}

Let us resolve indeterminacies of the map $f: \p^2 \dashrightarrow
\p^2$ via the blow-up of $I_f$; this is possible due to the
generality assumption (cf. {\ref{subsection:introd-2}}):

\begin{equation}
\label{diag} \xymatrix{
& X \ar[dl]_\varphi \ar[dr]^{\psi} \\
\p^2 \ar@{-->}[rr]^{f} && \p^2}
\end{equation}
Here $\varphi$ is the blow-up of $P_1,\ldots,P_{9 - \delta}$, so
that the birational transform $\varphi_*^{-1}(\Lambda)$ coincides
with the linear system $|-K_X|$ for the canonical class $K_X =
\varphi^*(K_{\p^2}) + E_1 + \ldots + E_{9 - \delta}$, where $E_i
:= \varphi^{-1}(P_i)$ are the exceptional divisors. It follows
that $\psi$ coincides with a \emph{regular} projection of the
degree $\delta$ smooth surface $X = X_\delta \subset \p^\delta =
\Lambda^*$ onto the plane $\p^2 = \Pi^*$.

Let us fix a point $o \in \p^2$. In order to show that $f^{-1}(o)
\ne \emptyset$ it suffices to check the condition $\psi^{-1}(o)
\not\subset E_i$ for all $1 \le i \le  9 - \delta$ (cf.
\cite[Proposition 3]{Kul-Zhd}). Here is the main result of this
section:

\begin{theorem}
\label{theorem:main} In the previous setting, we have
$\psi^{-1}(o) \not\subset E_i$ for all $i$ and $o$, so that the
map $f$ is surjective.
\end{theorem}

\begin{proof}
Let $\Sigma \subset \p^\delta$ be the center of the projection
$\psi$
--- this is a subspace of codimension $3$ and $\psi$ is
given by the linear system of all hyperplanes passing through
$\Sigma$. In particular, since $E_i \subset X$ is a line, its
image $\psi(E_i) \subset \p^2$ is also a line. Then we have
$\psi^{-1}\psi(E_i) = X \cap H = E_i + R$ for some hyperplane $H$
and a curve $R \sim -K_X - E_i$. It follows from the generality
assumption on $f$ that $H$ is a generic one. Then we claim the
following:

\begin{lemma}
\label{theorem:r-is-smooth} $R$ is smooth.
\end{lemma}

\begin{proof}
Consider the projection $\p^\delta \dashrightarrow \p^{\delta-2}$
from the line $E_i$. This rational map can be regularized via the
single blow-up $\nu: W \map \p^\delta$ of $E_i$. We have
$\nu_*^{-1}(X) \simeq X$ and the restriction of the obtained
morphism $W \map \p^{\delta-2}$ to $X \subset W$ coincides, by
construction, with the one given by the linear system $|R|$. Hence
we get $\text{Bs}(|R|) = \emptyset$. Then $R$ is smooth by
Bertini's theorem and generality of $H$.
\end{proof}

Further, note that $R \cdot E_i = 2$, which means $R \cap E_i =
\{o_1,o_2\}$ for some (distinct) points $o_j$. Now that $\psi$
maps the union $R \cup E_i$ onto $\psi(E_i)$, in order to prove
our Theorem it suffices to show the morphism $\psi\big\vert_R$ is
\emph{not ramified} at $o_1,o_2$ (cf.
Lemma~\ref{theorem:r-is-smooth} and Figure~\ref{fig-2}). But this
is immediate
--- again by generality of $f$ (and $\Sigma$).

Thus we obtain $\psi^{-1}(o) \not\subset E_i$ for all $i$ and $o$.
Theorem~\ref{theorem:main} is proved.

\bigskip

\begin{figure}[h]
\includegraphics[scale=1.4]{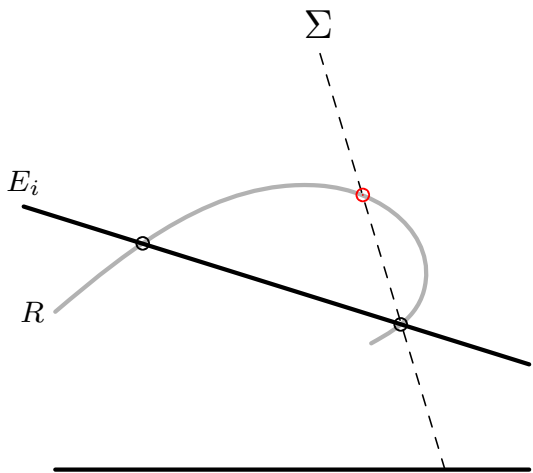}
\caption{Projection onto $\psi(E_i)$}\label{fig-2}
\end{figure}

\end{proof}

\bigskip

\section{Experiments and examples}
\label{section:ex-calc}

\refstepcounter{equation}
\subsection{Quartic case}
\label{subsection:ex-calc-1}

Here we consider a particular instance of (general cubic)
surjective maps $f$ with the locus $I_f$ consisting of five points
in general position. According to the reasoning outlined in
Section~\ref{section:prof} such maps correspond to certain
projections $X \map \p^2$ of a del Pezzo surface $X \subset \p^4$
of degree $\delta = 4$.

Specifically, let us form a linear system $\Lambda \subset
|-K_{\p^2}|$ of all cubics on $\p^2$, which pass through the
points $[1:0:0],\ [0:1:0],\ [0:0:1],\ [1:1:1],\ [2:3:1]$. Ideally,
given a plane $\Pi \subset \Lambda \simeq \p^4$, one wants an
\emph{explicit} way of checking whether the corresponding map $f:
\p^2 \dashrightarrow \Pi^*$ is surjective. In other words, one
wants to exclude (explicitly) all unruly pencils in $\Pi$, as we
have discussed in Remark~\ref{remark:surj-robust}. The most
straightforward way is to restrict everything to a finite field
$\mathbb{F}_p$ and list all the possibilities. Here is a
Python-Sage code doing this:\footnote{To lighten the calculations
we consider $\Lambda$ over $\mathbb{F}_7$ (cf. generators of the
ideal $I$ below).}

\medskip
{\small
\begin{lstlisting}

from sage.all import *

import numpy as np

# specify a finite field and algebraic data over it

k = GF(Integer(2)); V = VectorSpace(k,5); VV = VectorSpace(k,3)

R = k['x, y, z']; (x, y, z,) = R._first_ngens(3)

PP = ProjectiveSpace(Integer(2), k, names=('x', 'y', 'z',)); (x,
y, z,) = PP._first_ngens(3)

# ideal defining the set of points [1:0:0], [0:1:0], [0:0:1],

# [1:1:1], [2:3:1]

I = ideal(x^2*y + 3*y^2*z - x*z^2 - 3*y*z^2, x*y^2 + 3*y^2*z -
2*x*z^2 - 2*y*z^2, x^2*z - y^2*z + 2*x*z^2 - 2*y*z^2, x*y*z +
3*x*z^2 + 3*y*z^2, z*(x - y)*(y - x - z))

B = I.groebner_basis(deg_bound=Integer(3))

w = vector(B)

# compute number of irreducible component in a scheme s

def irred(s):
    ir_comp = s.irreducible_components()
    return len(ir_comp)

# create P2 in the given subspace of cubics

def make_plane(v, u, t):
    p = ideal(v.dot_product(w), u.dot_product(w), t.dot_product(w))
    return p

# number of base points in a given P2

def base_locus_length(v, u, t):
    some_plane = make_plane(v, u, t)
    base_loc = PP.subscheme(some_plane)
    return irred(base_loc)

# create a pencil in the given P2

def make_pencil(v, u, t, a, b):
    s_plane = make_plane(v, u, t)
    c_base = s_plane.groebner_basis(deg_bound=Integer(3))
    v_tran = vector(c_base)
    penc = ideal(a.dot_product(v_tran), b.dot_product(v_tran))
    return(PP.subscheme(penc))

# for given P2 inspect over all its pencils and return 0 if base

# loci of P2 and a pencil in it have the same length

def num_surjective(v, u, t):
    x = 0
    for a in VV:
        for b in VV:
            vari = make_pencil(v, u, t, a, b)
            if vari.dimension() > 0:
                        continue
            else:
                if irred(vari) == base_locus_length(v, u, t):
                            return 0
                            return
                else:
                    x = 1
                    continue
    return x

# auxiliary function: sum of squares of coordinates

def sq_coord(vect):
    s = sum(x**2 for x in vect)
    return s

# auxiliary function: check whether a collection of three vectors

# is orthonormal => restrict to such v, u, t below

def is_orth(vectors):
    for i in range(3):
        if sq_coord(vectors[i]) != 1:
               return False
    for j in range(i + 1, 3):
        if np.dot(vectors[i], vectors[j]) != 0:
            return False
    return True

# list all dominant maps and print them into file

dic_val = {}

for v in V:
    for u in V:
         for t in V:
                var = PP.subscheme(make_plane(v, u, t))
                if matrix([v, u, t]).rank() < 3 or var.dimension() > 0:
                                    continue
                else:
                    if __name__ == "__main__":
                               vectors = np.array([tuple(v), tuple(u), tuple(t)])
                               if is_orth(vectors) == False:
                                        continue
                               else:
                                     dic_val[tuple(v), tuple(u), tuple(t)] =
                                     num_surjective(v, u, t)


filename = 'output.txt'

with open(filename, 'w') as file:
    for key, value in dic_val.items():
        file.write(f"{key}: {value}\n")

print(f"Values have been written to {filename}")

\end{lstlisting}
}

\medskip

The program lists all vectors {\texttt{v, u, t}} from $\Lambda$,
spanning various planes $\Pi$ as above, and writes the result to
the file \texttt{output.txt} in the form \texttt{(v, u, t)}:
\texttt{value}. Here \texttt{value} equals either ${\bf 1}$ or
${\bf 0}$, depending on whether the corresponding map $f$ is
surjective or not, respectively.

To reduce the amount of calculations we inspect only those vectors
$\texttt{v, u, t}$ that form an orthonormal basis in $\Pi$ (cf.
the function $\verb"is_orth"$). However, even under this
restriction our program was able to generate \texttt{output.txt}
only for $p = 2$, as for larger $p$ the computation becomes
unreasonably long.\footnote{In this respect, it is worth
developing an algorithm, which will not distinguish between
\emph{projectively equivalent} triples $\texttt{(v, u, t)}$.} But
still, these experimental results over $\mathbb{F}_2$ allow one to
obtain non-trivial examples of surjective $f$, now defined over
$\com$.

Namely, let us turn to the case no. {\bf 46} in
\texttt{output.txt}, with $\texttt{(v, u, t)}$ being
$((1,0,0,0,0),\ (0,0,0,1,0),\ (1,1,0,0,1))$ and \texttt{value}
equal ${\bf 1}$. The linear system $\Lambda$ is spanned by the
cubics
$$
x^2y + y^2z,\ xy^2 + y^2z,\ x^2z + y^2z,\ xyz,\ xz^2 + yz^2
$$
(cf. vector \texttt{w} in the above code) --- please mind that all
this is over $\mathbb{F}_2$.\footnote{The shape of the components
of \texttt{w} is more complicated when $p \ge 5$.} Now let us
(formally) extend $\Lambda$ with the present basis to be defined
over $\com$.

Thus we are considering $f: \p^2 \dashrightarrow \p^2$ given by
$$
[x: y: z] \mapsto [x^2y + y^2z: xyz: x^2y + xy^2 + 2y^2z + z^2(x +
y)].
$$
Note that $I_f = \text{Bs}(\Lambda) = \{[1:0:0],\ [0:1:0],\
[0:0:1]\}$, i.e. the points $[1:1:1]$ and $[2:3:1]$ got lost in
the way, as results from the mod $2$ reduction and extension to
$\com$ procedure. We claim that

\begin{prop}
\label{theorem:4-dp-surj} The given map $f$ is surjective.
\end{prop}

\begin{proof}
Put $P := [1:0:0]$. Then it is immediate that $f^{-1}(P) \ni
[0:1:-2]$. Next we take $Q := [0:0:1]$ and observe that $f^{-1}(Q)
\ni [1:0:1]$.

Further, consider the line $L = (y = 0) \ni P,Q$ and pick an
arbitrary point $[1:0:a] \in L \setminus\{P,Q\}$. Then it is
immediate that $f([0:1:a - 2]) = [1:0:a]$. Thus the map $f$ is
\emph{surjective over $L$}.

Let us now verify surjectivity over the complement $\p^2 \setminus
L$. Any point in the latter is of the form $[a:1:b]$ for some $a$
and $b \in \com$. Hence we must have $xyz \ne 0$. Then, letting $z
= 1$, it suffices to show that the point $(a,b) \in \com^2$ has a
preimage under the map
$$
(x,y) \mapsto (x + \frac{y}{x}, x + y + \frac{2y}{x} + \frac{1}{x}
+ \frac{1}{y})
$$
(we have divided the components of $f$ by $xy$). The condition $x
+ \displaystyle\frac{y}{x} = a$ gives $y = ax - x^2$ (with
non-zero $x$) and for the second component we get
$$
x^2y + xy^2 + 2y^2 + x + y = bxy.
$$
After substituting the expression for $y$ and dividing by $x$ we
obtain the equation
\begin{equation}
\label{eq-x-1} x^2(a - x) + x^2(a - x)^2 + 2x(a - x)^2 + 1 + a - x
- bx(a - x) = 0.
\end{equation}
Note that $x = 0$ forces $a = -1$. Next we bring \eqref{eq-x-1} to
the form
$$
x^4 + (1 - 2a)x^3 + (a^2 - 3a + b)x^2 + (2a^2 - ba - 1)x + a + 1 =
0.
$$
For $a \ne -1$, the latter equation has a non-zero solution, and
so $(a,b)$ has a preimage by our construction. Assume that $a =
-1$. Then the equation becomes (after dividing by $x$)
$$
x^3 + 2x^2 + (b + 4)x + b + 1 = 0.
$$
Again, if $b \ne -1$, then we are done. Finally, in the case of $b
= -1$ we get quadratic equation $x^2 + 2x + 3 = 0$, which has a
non-zero solution. This completes the proof that $f$ is
surjective.
\end{proof}

\refstepcounter{equation}
\subsection{Entering ML}
\label{subsection:ex-calc-2}

The above reasoning aims to illustrate a general strategy for
finding non-trivial surjective rational endomorphisms of $\p^2$:
make predictions over (small) finite field (cf. file
\texttt{output.txt}) and then inspect those $\texttt{v, u, t}$,
which ``likely give'' surjective $f$ (cf. \texttt{value = {\bf 1}}
condition). The latter requires, as
Proposition~\ref{theorem:4-dp-surj} reveals, quite tedious
(although direct) calculations. For the former, a seemingly
important question is:
$$
\text{how to make ``predictions of surjectivity'' for
\emph{arbitrary} vectors}\ \texttt{v, u, t},\ \text{not
necessarily over}\ \mathbb{F}_p?
$$
One way to answer this, as became common nowadays, is to generate
a ``prediction function'' via neural network by inspecting the
table from \texttt{output.txt} (cf. the discussion in
{\ref{subsection:introd-3}}). Here is the corresponding code:

\medskip

{\small
\begin{lstlisting}

import ast import numpy as np

import tensorflow as tf

import matplotlib.pyplot as plt

from sklearn.model_selection import train_test_split

from sklearn.preprocessing import StandardScaler

from keras.models import Sequential

from keras.layers import Dense, Conv2D, Flatten

# create a dictionary from the given file

def read_dict_file(filename: str) -> dict:
    result = {}
    with open(filename, 'r') as file:
        for line in file:
            key, value = line.strip().split(':')
            result[ast.literal_eval(key)] = int(value)
    return result

if __name__ == "__main__":
    filename = 'output.txt'
    my_dict = read_dict_file(filename)

# function we are going to learn

def func_from_dict(u, v, w):
    return my_dict[u, v, w]

# prepare data (not scaled)

X_first = np.array(list(my_dict.keys()))

Y = np.array(list(my_dict.values())).reshape(-1, 1)

# scale data

scaler_Y = StandardScaler()

Y_train = scaler_Y.fit_transform(Y_train)

Y_test = scaler_Y.transform(Y_test)

X_list = []

for matrix in X_first:

    matrix_reshaped = matrix.reshape(-1,matrix.shape[-1])
    scaler_X = StandardScaler()
    scaled_matrix = scaler_X.fit_transform(matrix_reshaped)
    scaled_matrix = scaled_matrix.reshape(matrix.shape)
    X_list.append(scaled_matrix)

X = np.array(X_list)

# split into train and test sets

X_train, X_test, Y_train, Y_test = train_test_split(X, Y,
test_size=0.2, random_state=42)

# build neural network

model = Sequential() model.add(Conv2D(256, kernel_size=(2,2),
activation='relu', input_shape=(3,5,1))) model.add(Flatten())
model.add(Dense(256, activation='relu')) model.add(Dense(1))

# compile model

model.compile(optimizer='adam', loss='mean_squared_error',
metrics=['accuracy'])

# model summary

model.summary()

# train model

model.fit(X_train, Y_train, epochs=150, batch_size=32)

# evaluate model

loss, accuracy = model.evaluate(X_test, Y_test)

print(f'test loss: {loss}')

print(f'test accuracy: {accuracy}')

# make predictions

Y_learn = model.predict(X_test)

Y_learn =scaler_Y.inverse_transform(Y_learn)

print(f'predictions: {Y_learn}')

\end{lstlisting}
}

\medskip

Our model is based on the Keras API. The NN's configuration (one
of many possible) includes the input layer with $256$ neurons and
one hidden layer with $256$ neurons (plus a convolution layer).
After training the model the pertinent prediction function
$Y_{\text{learn}}$ is obtained.\footnote{We did not intend to be
in any way optimal here and the test loss (resp. accuracy) turns
out to be $0.078$ (resp. $0.938$). We also do not care for any
sort of interpretability (but of course, applying KAN neural
network, say, in our setting seems very intriguing --- see
\cite{KAN}).} Its average against the test data can be found:

\begin{lstlisting}

s = Y_learn.sum()

t = len(X_test)

print(s/t)

\end{lstlisting}
and is equal to $0.0737$. These calculations reveal a curious fact
that over the field $\mathbb{F}_2$ (over which \texttt{output.txt}
has been generated) the analog of Theorem~\ref{theorem:main}
\emph{does not hold}: for average rational map $f$ (aka
\texttt{(v, u, t)}), its $Y_{\text{learn}}$-value is close to $0$.

Thus with \emph{any} $f$ one may associate its \emph{measure of
surjectivity} $Y_{\text{learn}}(f)$ and ask to what extent does
the latter reflect the corresponding geometric property of the
given map. The preceding computer experiments, as well as
Proposition~\ref{theorem:4-dp-surj} and its proof (see also
{\ref{subsection:ex-calc-3}}), show that this is a difficult
problem in general.

\refstepcounter{equation}
\subsection{Cubic case}
\label{subsection:ex-calc-3}

Let us add one more point $[3:2:1]$ to the base locus of $\Lambda$
from {\ref{subsection:ex-calc-1}}. This corresponds to erasing the
term \texttt{z*(x - y)*(y - x - z)} in the ideal \texttt{I} from
the above Python-Sage code.\footnote{Once again everything is over
$\mathbb{F}_7$ for simplicity.} Now, after running the script
(with \texttt{V = VectorSpace(k,4)}), the file \texttt{output.txt}
is generated in a reasonable time again only for $p = 2$. One can
see that now \texttt{value} equals ${\bf 0}$ \emph{for all}
triples \texttt{(v, u, t)}. Thus $Y_{\text{learn}}(f) = 0$ for all
$f$, i.e. they are ``close to be non-surjective'', according to
the heuristic principle in {\ref{subsection:ex-calc-2}}. Yet,
amazingly, one can find surjective maps among these $f$. Let us
proceed as earlier.

Namely, the basis of $\Lambda$ now consists of cubics
$$
x^2y + y^2z + z^2(x + y),\ y^2(x - z),\ z(x^2 + y^2),\ xyz + xz^2
+ yz^2,
$$
and we consider the triple $((1,0,0,0), (0,1,0,0), (0,0,0,1))$.
Then the corresponding map $f: \p^2 \dashrightarrow \p^2$ is given
by
$$
[x: y: z] \mapsto [x^2y + y^2z + z^2(x + y): y^2(x - z): xyz +
z^2(x + y)].\footnote{Here, as in {\ref{subsection:ex-calc-1}},
all data is considered over $\com$. Hence we distinguish $y^2(x -
z)$ and $y^2(x + z)$ from the initial $\mathbb{F}_2$-basis of
$\Lambda$.}
$$
Again we have $I_f = \text{Bs}(\Lambda)$ (over both $\mathbb{F}_2$
and $\mathbb{F}_7$).

\begin{prop}
\label{theorem:3-dp-surj} The given $f$ is surjective.
\end{prop}

\begin{proof}
Put $P := [1:0:0]$. Then it is immediate that $f^{-1}(P) \ni
[-2:1:-2]$. Next we take $Q := [0:0:1]$ and observe that
$f^{-1}(Q) \ni [-1:1:-1]$.

Further, consider the line $L = (y = 0) \ni P,Q$ and pick an
arbitrary point $[a:0:1] \in L \setminus\{P,Q\}$. Then for
$[x:y:z] \in f^{-1}([a:0:1])$ we get the following conditions:
$$
x = z = 1,\quad \frac{(y + 1)^2}{2y + 1} = a,\ y \ne -0.5.
$$
This yields quadratic equation $y^2 + 2(1 - a)y + 1 - a = 0$. The
latter always has a solution $\zeta \ne -0.5$. In particular, we
obtain that $[1:\zeta:1]$ is contained in $f^{-1}([a:0:1])$, and
hence the map $f$ is \emph{surjective over $L$}.

Let us now verify surjectivity over the complement $\p^2 \setminus
L$. Any point in the latter is of the form $[a:1:b]$ for some $a$
and $b \in \com$. Hence we must have $y \ne 0$. Then, letting $y =
1$, it suffices to show that the point $(a,b) \in \com^2$ has a
preimage under the map
$$
(x,z) \mapsto (\frac{x^2 + z + xz^2 + z^2}{x - z}, \frac{xz + xz^2
+ z^2}{x - z})
$$
or, after transforming $(a,b) \mapsto (a - b, b)$, under the map
$$
(x,z) \mapsto (\frac{x^2 + z - xz}{x - z}, \frac{xz + z + xz^2 +
z^2}{x - z}).
$$
The condition $\displaystyle\frac{x^2 + z - xz}{x - z} = a$ gives
$z = \displaystyle\frac{ax - x^2}{1 + a - x}$ (with $x \ne 1 + a,
z$), and for the second component we get
$$
xz + z + xz^2 + z^2 = b(x - z).
$$
Now, substituting, dividing by $x$ and clearing denominators gives
the equation
\begin{equation}
\label{q-eq-zer} x^4 + (-2a + 2)x^3 + (a^2 - 4a)x^2 + (2a^2 - a +
b - 1)x + a^2 + a(1 - b) - b = 0.
\end{equation}

Note that the condition $z = \displaystyle\frac{ax - x^2}{1 + a -
x} = x$ forces $x = 0$. Quartic polynomial in \eqref{q-eq-zer} has
only this value as its root iff it coincides with $x^4$. This
gives two incompatible relations $2a - 2 = 0 = a^2 - 4a$ (for the
coefficients at $x^3$ and $x^2$). Hence we have verified that $x
\ne z$.

It remains to deal with the case when possibly $x = 1 + a$. After
substituting in \eqref{q-eq-zer} one gets $a^2 + 3a + 2 = 0$. On
the other hand, from the expression $\displaystyle\frac{x^2 + z -
xz}{x - z} = a$, with $x = 1 + a$, we obtain $a = -1$, a
contradiction. This completes the proof that $f$ is surjective.
\end{proof}

\bigskip

\bigskip

\thanks{{\bf Acknowledgments.}
I am grateful to V. Sakbaev for giving me an opportunity to
present the results of this paper. The work was carried out at the
Center for Pure Mathematics (MIPT) and was partially supported by
the Russian Science Foundation under grant \textnumero\,
25-21-00083 (https://rscf.ru/project/25-21-00083/).

\end{document}